\documentclass[10pt, reqno]{amsart} 

\usepackage{fullpage}
\usepackage[margin=1in]{geometry}
\usepackage{amsmath,amssymb,amsthm,enumerate}
\usepackage{mathtools}
\usepackage{color}

\usepackage{scalerel}
\usepackage{stackengine}

\stackMath
\newcommand\reallywidehat[1]{%
\savestack{\tmpbox}{\stretchto{%
  \scaleto{%
    \scalerel*[\widthof{\ensuremath{#1}}]{\kern.1pt\mathchar"0362\kern.1pt}%
    {\rule{0ex}{\textheight}}
  }{\textheight}%
}{2.4ex}}%
\stackon[-6.9pt]{#1}{\tmpbox}%
}

\numberwithin{equation}{section}
\newtheorem{thm}{Theorem}[section]

\newtheorem{lemm}[thm]{Lemma}
\newtheorem{prop}[thm]{Proposition}

\newtheorem{rem}[thm]{Remark}
\newtheorem{defn}[thm]{Definition}

\newcommand{\bke}[1]{\left ( #1 \right )}

\newcommand{\norm}[1]{ \| #1  \|}

\newcommand\al{\alpha}
\newcommand\be{\beta}
\newcommand\ga{\gamma}

\newcommand\ep{\epsilon}

\renewcommand\th{\theta}

\newcommand\la{\lambda}

\newcommand\De{\Delta}

\newcommand\La{\Lambda}

\newcommand{\R}{\mathbb{R}}

\newcommand{\Z}{\mathbb{Z}}
\newcommand{\N}{\mathbb{N}}

\newcommand{\pa}{\partial}
\newcommand{\na}{\nabla}

\renewcommand{\bar}[1]{\overline{#1}}
\newcommand{\lec}{{\ \lesssim \ }}

\newcommand{\cR}{\mathcal{R}}

\newcommand{\EQ}[1]{\begin{equation}\begin{split} #1 \end{split}\end{equation}}
\newcommand{\EQN}[1]{\begin{equation*}\begin{split} #1 \end{split}\end{equation*}}

\newcommand{\simp}{\ \overset{\circ}{\approx}\ }
\newcommand{\T} { {\mathbb{T}^2} }
\newcommand{\len}{{\leq n}}

\providecommand{\myceil}[1]{\left \lceil #1 \right \rceil }

\newcommand{\smb}[1]{\scalebox{0.85}{#1}}
\newcommand{\bmt}[1]{\boxed{\operatorname{#1}}}

\title{Non-uniqueness of steady-state  weak solutions to \\ the surface quasi-geostrophic equations}
\author{Xinyu Cheng,  Hyunju Kwon, Dong Li}

\address{
\begin{minipage}{\linewidth}
Xinyu Cheng, University of British Columbia, \texttt{xycheng@math.ubc.ca}\\ \\
Hyunju Kwon, Institute for Advanced Study, \texttt{hkwon@math.ias.edu}\\  \\
Dong Li, Hong Kong University of Science and Technology, \texttt{madli@ust.hk} 
\end{minipage}
} 

\begin{document}

\maketitle
\begin{abstract}
We show the existence of nontrivial stationary weak solutions to the surface quasi-geostrophic
equations on the two dimensional periodic torus. 
\end{abstract}

\section{Introduction}
 Consider the 2D surface quasi-geostrophic (SQG) equations 
  for  $\th=\theta(x, t) : \T\times [0,\infty) \to \R$: 
\begin{equation}\label{sqg}
\begin{cases}
\pa_t \th + u \cdot \na \th = -\nu \La^{\ga} \th,   \qquad \text{in $\T \times(0,\infty)$;}     \tag{\text{SQG}}\\
u= \na^{\perp}\La^{-1}\th=(-\partial_2 \La^{-1} \th, \partial_1 \La^{-1} \th) =
(- \mathcal R_2 \theta,  \mathcal R_1 \theta ); \\
\th|_{t=0}= \th_0,
\end{cases}
\end{equation}
where $\nu\ge 0$ is the viscosity, $0<\gamma \le 2$
and $\mathbb T^2=[-\pi, \pi]^2$ is the periodic torus. For $s\ge 0$ the fractional Laplacian $\La^s=(-\Delta)^{\frac s2}$
is defined by (under suitable assumptions on $\theta$)
$
\widehat{ \La^s \theta} (k) = |k|^{s} \hat \theta (k)
$
for $k\in \mathbb Z^2$. 
For negative $s$ the formula is restricted to nonzero wave numbers.  We consider solutions with zero mean which is invariant under the dynamics thanks to incompressibility.   Note that the operators $\mathcal R_j$, $j=1,2$ are
skew-symmetric,
 i.e. $\langle  \mathcal R_j f, g \rangle = - \langle f ,  \mathcal R_j g \rangle $ where $\langle, \rangle$ denotes
the usual $L^2$ (real) inner product. Using this one can derive for $\theta \in L^2$
(below $[A,B]=AB-BA$ is the usual commutator):
\begin{align*}
\langle \theta  \mathcal R_j  \theta, \phi\rangle = -\frac 12 \langle \theta, [\mathcal  R_j, \phi] \theta\rangle,
\qquad \forall\, \phi \in C^{\infty}(\mathbb T^2).
\end{align*}
 By Proposition \ref{prop_smoot1} one has
$
\| [R_j, \phi ] \theta \|_{\dot H^{\frac 12} } \lesssim \| \phi \|_{H^3} \| \theta \|_{\dot H^{-\frac 12}}
$
and thus $\dot H^{-\frac 12}$ regularity suffices for defining a weak solution. To state our main theorem (Theorem \ref{thm}), 
we introduce a definition.
\begin{defn}
We say $\th\in \dot{H}^{-\frac 12}(\T)$ ( $\overline{ \theta}=0$) is 
a stationary weak solution  to \eqref{sqg} if 
\begin{align*}
\frac 12 \int_{\mathbb T^2}
(\Lambda^{-\frac 12 } \th)  \cdot \La^{\frac 12} ( [\mathcal R^{\perp}, \nabla \psi] \theta ) dx
= - \nu \int_{\mathbb T^2} (\La^{-\frac 12 } \theta) \La^{\gamma+\frac 12} \psi dx, \quad\, \forall\, \psi \in C^{\infty}(\mathbb T^2), 
\end{align*}
where $[\mathcal R^{\perp}, \nabla \psi] \theta = - [ \mathcal R_2, \partial_1 \psi] \theta + [\mathcal R_1, \partial_2 \psi] \theta$.

\end{defn}

\begin{thm}\label{thm}
For any $\nu\ge 0$, $\ga\in(0,\frac 32)$, and $\frac 12 \le  \al<\frac 12 +\min(\frac 1{6}, \frac 32-\ga)$, there exist infinitely many 
nontrivial steady-state/stationary weak solutions $\th$ to \eqref{sqg} satisfying  $\overline{\theta}\equiv 0$ and 
$\La^{-1}\th \in C^{\alpha}(\mathbb{T}^2)$. 
\end{thm}
\begin{rem}
For the non-steady case,   by employing time-dependent test functions, 
 one can define weak solutions in  \scalebox{0.85}{$ L^2_{t,\operatorname{loc} } \dot H_x^{-\frac 12}$}.  Resnick \cite{Res95}  proved the global existence of a weak solution to \eqref{sqg} for $\nu\geq 0$ and 
 \scalebox{0.85}{$0< \ga\le 2$} in \scalebox{0.85}{$L_t^{\infty}L_x^2$} for any initial data 
 \scalebox{0.85}{$\th_0\in L^2(\T)$}. 
 Marchand \cite{Mar08} obtained a global  weak solution in \scalebox{0.85}{$L_t^{\infty} H^{-\frac 12}$}  for  \scalebox{0.85}{$\th_0\in \dot{H}^{-\frac 12}(\mathbb R^2)$} or \scalebox{0.85}{$L_t^{\infty} L^p$} for  \scalebox{0.85}{$\th_0 \in L^p_x(\mathbb R^2)$}, 
 \scalebox{0.85}{$p\ge \frac 43$}, when\footnote{For \smb{$\nu=0$} one requires \smb{$p> 4/3$}
 since the embedding \smb{$L^{\frac 43} \hookrightarrow \dot H^{-\frac 12}$} is not compact.
 For the diffusive case one has extra \smb{$L_t^2 \dot H^{\frac {\gamma}2-\frac 12}$} conservation
 by construction.    }
  \scalebox{0.85}{$\nu>0$} and \scalebox{0.85}{$0<\gamma\le 2$}. 
Non-uniqueness of weak solutions satisfying \scalebox{0.85}{$\| \La^{-\frac 12} \theta(t) \|_2 =e(t) $} (for any prescribed  
\scalebox{0.85}{$0\le e\in C_c^{\infty}$})
was obtained in \cite{BSV16} in two cases:
\scalebox{0.85} {$ 
\nu=0:\;  \frac 12<\beta<\frac 45, \, \sigma<\frac{\beta}{2-\beta}  ,\; \La^{-1}\theta \in C_t^{\sigma} C_x^{\beta}$} and
\scalebox{0.85}{$
 \nu>0:  \;\frac 12 <\beta<\frac 45, \, 0<\gamma<2-\beta, \,  \sigma<\frac{\beta}{2-\beta},
\;  \La^{-1}\theta \in C_t^{\sigma} C_x^{\beta}. 
$} Note that the restriction \scalebox{0.85}{$\beta-1 <1-\gamma$} accords with the critical 
\scalebox{0.85}{$\| \theta\|_{L_t^{\infty} \dot C^{1-\gamma}}$} norm. For \scalebox{0.85}{$\nu=0$} by using the identity
\scalebox{0.85}{$\frac 12 \partial_t ( \| \La^{-\frac 12} P_{<J} \th \|_2^2)
=-\int P_{<J} (\th \mathcal R^{\perp} \th) \cdot P_{<J} \mathcal R \th dx$}, one has conservation of
\scalebox{0.85}{$\|\La^{-\frac 12}
\th\|_2^2$} for \scalebox{0.85}{$\th \in L_{t,x}^3$} (see also \cite{IsVi15}). 
We also mention that for the non-dissipative case in the positive direction some uniqueness
of SQG patches with some regularity for the moving boundary were obtained
in recent \cite{CCG18}.
\end{rem}
\begin{rem}
Consider a general active scalar \scalebox{0.85}{$\partial_t \theta+\nabla \cdot ( \theta u)=0$} where 
\smb{$\widehat{u}=m(k) \widehat{\theta}(k)$}.  By using a plane wave ansatz \smb{$\theta = a_k e^{i \la k \cdot x}
+a_{k}^{*} e^{-i \la k\cdot x}$} with  \smb{$|k|=1$ and $\la\gg 1$}, we derive  
\smb{$ \nabla \cdot ( \theta u) 
\underset{\text{low freq}}
\approx \nabla \cdot ( |a_k|^2 ( m(-\lambda k) +m(\lambda k) ) ) $ } which vanishes 
if $m$ is odd. This is known as the odd multiplier obstruction \cite{DeSz12, Shv11, IsVi15} in applying the convex integration scheme
\cite{DeSz19}.  Previously the non-uniqueness results were established only for active scalar equations with non-odd multipliers \cite{Shv11,IsVi15}. 
In \cite{BSV16} this issue was resolved by using the momentum equation for $v= \La^{-1} u$  and rewriting the nonlinearity
\smb{$ u\cdot \nabla v - (\nabla v)^{T} \cdot u$} as the sum of a divergence of a $2$-tensor, and a gradient of a scalar function.  Recently Isett and Ma
\cite{IsMa20} gives another direct approach at the level of $\theta$.

\end{rem}
\begin{rem}
The restriction\footnote{This restriction is immaterial in some sense since $\gamma>1$ is
subcritical.}
 $\gamma<\frac 32$ in Theorem \ref{thm} can be seen by a crude heuristic using
the plane wave ansatz localized around frequency $\lambda$. The domination of 
nonlinearity versus dissipation yields $\| \Lambda^{-1} \theta\|_{\infty} \gg \lambda^{\gamma-2}$. 
H\"older regularity of $\Lambda^{-1} \theta$ yields $\|\Lambda^{-1}\theta\|_{\infty}
\lesssim \lambda^{-\alpha}$ where $\alpha>\frac 12$.  Thus $\gamma\le 2-\alpha<\frac 32$.
\end{rem}
 
The modest goal of  this work  is to introduce another approach to overcome the odd multiplier obstruction by working directly with the scalar function $f= \La^{-1}\th$.  Returning to the plane
wave ansatz, the SQG nonlinearity written for $f$ is $Q^{\nabla} (\La f \nabla^{\perp} f)$ where $Q^{\nabla}$ means projection to the gradient
direction.\footnote{We caution the reader that this is slightly different than the usual convex integration scheme in Euler.}  
Now consider $f= \sum a_l(x) \cos (\lambda l\cdot x)$ where $|l|=1$ and $\lambda \gg 1$,  then (see Lemma \ref{Leib})
\[
\La f = \la f + (l\cdot \na) a \sin(\la l\cdot x) + (T^{(1)}_{\la l}a) \cos (\la l\cdot x) + (T^{(2)}_{\la l}a)\sin (\la l\cdot x).
\]
Thus we have
\begin{align*}
\La f \nabla^{\perp} f \simp  - \frac 14\la \sum_{l} (l\cdot \na)(a_l^2) l^{\perp} + \text{error terms}.
\end{align*}
We then use a novel algebraic lemma (Lemma \ref{decomp.id}) to obtain nontrivial projection in the gradient direction. One should note  that in the
above computation,  the leading $O(\lambda^2)$ term vanishes which completely accords with the odd multiplier obstruction problem
mentioned earlier. 
What is remarkable is  that in the next $O(\lambda)$ term there is nontrivial non-oscillatory contribution coming from the commutator piece
$[\La, a_l] \cos \lambda x$. This resonates with the momentum approach in \cite{BSV16} and also the recent work
\cite{IsMa20}.

Our next result is about the weak rigidity of solutions in the time-dependent case. It improves
Theorem 1.3 of \cite{IsVi15} from $L_t^p L_x^2$, $p>2$ to $L_t^2 \dot H^{-\frac 12+}$.
 The proof can be found in Appendix B. 
\begin{thm}[Weak rigidity] \label{thm2}
Let $\nu\ge 0$ and $0<\gamma\le 2$. 
Suppose $f = \lim_n \theta_n$ is a weak limit of solutions \eqref{sqg} in $L_t^2 \dot H^s$ for $s>-\frac 12$. Then $f$ must also be a weak solution.
\end{thm}

\subsection*{Notations}
For any two quantities \smb{$A$} and \smb{$B$}, \smb{$A\lesssim B$} denotes 
\smb{$A\leq CB$} for some absolute constant \smb{$C>0$}. Similarly, 
\smb{$A\gtrsim B$} means \smb{$A \geq CB$}, and \smb{$A\sim B$} when 
\smb{$A\lesssim B$} and \smb{$A\gtrsim B$}. 
For a real number \smb{$X$}, we use \smb{$X^+$} for \smb{$X+\ep$} when $\ep>0$ is sufficiently small. For any two vector functions $v$ and $w$, we denote
\begin{center}
\boxed{v\simp w, \quad \text{if}\quad v=w+\nabla^{\perp} p}
\end{center}
holds for some smooth scalar function $p$. The mean of  $f$ on $\T$ is denoted by $\overline{f} =  \frac{1}{(2\pi)^2}\int_{\mathbb{T}^2} f(x) dx$. The function space 
\smb{$C_0^\infty(\T)$} is the collection of all \smb{$C^{\infty}$} mean-zero functions on $\T$.
For any 
\smb{$1\leq p\leq \infty$}, we denote \smb{$\| f\|_p=\norm{f}_{L^p(\T)}$} as the usual Lebesgue norm. 
 For $f$ on $\T$, we follow the Fourier transform convention
\smb{$
\hat{f}(k) = \frac 1{(2\pi)^2} \int_{\T} f(x) e^{-ix\cdot k} dx$} and \smb{$
f(x) = \sum_{k\in \Z^2} \hat{f}(k) e^{ik\cdot x}.
$}
The convolution operation $*$ is defined by
\smb{$
(f* g)(x) = \frac 1{(2\pi)^2} \int_{\T} f(x-y)g(y) dy,
$}
which implies 
\smb{$
\widehat{f*g}(k) = \hat{f}(k)\hat{g}(k)$} and \smb{$
\widehat{fg}(k) = \sum_{l\in \Z^2}\hat{f}(l)\hat{g}(k-l)$. } 

For $s\in \mathbb R$, the homogeneous \smb{$\dot H^s$} 
Sobolev norm is defined by
\smb{$
\norm{f}_{\dot{H}^s(\T)} = \left(\sum_{0 \ne k \in \mathbb Z^2} |k|^{2s}|\hat{f}(k)|^2 \right)^{\frac 12}.
$ }

\section{Proof of Theorem \ref{thm} }
Note that for  $f= \La^{-1} \th \in C_0^{\infty}(\T)$ the steady-state SQG equation is 
$\na\cdot\left(\La{f}\na^{\perp}{f}\right)=-\nu\La^{\ga+1}{f}$ 
which follows from 
$
\La {f}\na^{\perp}{f}
\simp \nu\La^{\ga-1}\na{f}.
$ The idea is to find approximate solutions
\smb{$(f_\len,q_n)\in C_0^\infty(\T)\times C_0^\infty(\T)$} solving the relaxed equation
\EQ{\label{eq.app}
\La f_\len\na^{\perp}f_\len\simp\nu\La^{\ga-1}\na f_\len+\na q_n,}
such that $q_n \to 0$ in the limit.  This will be done inductively. 
\begin{proof}[Proof of Theorem \ref{thm}]
WLOG we take $C_0=2$ in Proposition \ref{matching}. 
Fix $\nu\ge 0$, $0< \ga<\frac 32$ and choose parameters as in \eqref{par}. Choose $b$ and $\la_0$ as in Proposition \ref{const.qn}. If necessary, we choose larger $\la_0$ to have $\sum_{m=0}^\infty\la_m^{\alpha-\frac 12-\frac{\beta}{2b}}\le 1$.  Take the base step $(f_{\le 0}, q_0)=(0,0)$. For $n^{\operatorname{th}}$-step, assume that
 $(f_\len, q_n) \in C_0^\infty(\T)\times C_0^\infty(\T)$ satisfy
\begin{itemize}
\item $(f_\len, q_n)$ solves \eqref{eq.app}.
    \item $\operatorname{supp}(\widehat{f_{\len}}) \subset\{|k| \le 6 \la_n\}$,
    $\operatorname{supp}(\widehat{q_{n}}) \subset\{|k| \le 12 \la_n\}$  and satisfies
    $\norm{q_n}_X\le r_n$ (see \eqref{X_def}), 
    \EQN{\norm{f_\len}_{C^\al(\T)}\leq 50
    \text{\smb{$\sum_{m=1}^n\la_m^{\al}\sqrt{\frac{r_{m-1}}{\la_m}}$}}\leq100 \sum_{m=0}^{n-1} \la_{m+1}^{\alpha-\frac 12-\frac{\beta}{2b} } \leq 100.
}
\end{itemize}
Then for step $n+1$, we  find $f_{n+1} = f_{\le n+1}-f_\len$ and $q_{n+1}\in
C_0^{\infty}(\T)$ satisfying
\begin{itemize}
    \item $f_{n+1}$ is chosen to  solve $\La f_{n+1}\na^{\perp}f_{n+1}+\na q_n \simp 0$  up
to a sufficiently small correction (see Proposition \ref{matching}).  Also
\smb{$\operatorname{supp}(\widehat{f_{\le n+1}}) \subset\{|k| \le 6 \la_{n+1}\}$}
and \smb{$\|f_{n+1} \|_{C^{\alpha}(\mathbb T^2)}
\le 100 \lambda_{n+1}^{\alpha}\cdot 
\sqrt{\frac{r_n }{\lambda_{n+1}} }$}.
    \item $\|q_{n+1}\|_X \le r_{n+1}$,  $\operatorname{supp}(\widehat{q_{n+1}}) 
    \subset\{|k| \le 12 \la_{n+1}\}$  and solves
         (see Proposition \ref{const.qn}) 
\EQ{\label{eq.pert}
(\La f_{n+1}\na^{\perp}f_{n+1}+\na q_n)
+\La f_\len\na^{\perp}f_{n+1}&+\La f_{n+1}\na^{\perp}f_\len
\simp\na q_{n+1}+\nu\La^{\ga-1}\na f_{n+1}.
}
\end{itemize}
Thus with the help of Proposition \ref{matching} and \ref{const.qn} the induction step can be closed and it remains to show that $f_{\le n}$
converges to the desired weak solution. We first check its regularity. Clearly 
\begin{align*}
\norm{f_{\le n'}-f_\len}_{C^\alpha}
\lesssim \sum_{m=n}^{n'-1} \la_{m+1}^{\alpha-\frac 12-\frac{\beta}{2b} },  \quad \forall\, n'\geq n.
\end{align*}
Thus $f_\len \to f\in C^{\al}(\T)$.  Now denote $\theta_n =\Lambda f_n$ and
$\theta=\Lambda f$.
Clearly 
\begin{align*}
\langle \theta_n \Lambda^{-1} \nabla^{\perp} \theta_n
-\nu \Lambda^{\gamma-2} \nabla \theta_{n+1} -\nabla q_{n+1}, \nabla \psi \rangle =0,
\quad \forall\, \psi \in C^{\infty}(\mathbb T^2).
\end{align*}
We then rewrite the above as
\begin{align*}
\frac 12 \langle \Lambda^{-\frac 12} \theta_n, \Lambda^{\frac 12}[\mathcal R^{\perp},\nabla
\psi]\theta_n \rangle+ \nu \langle {\La}^{-\frac 12} \theta_n, \Lambda^{\gamma+\frac 12} \psi\rangle
+ \langle q_n, \Delta \psi \rangle =0, \quad\forall\, \psi \in C^{\infty}(\mathbb T^2).
\end{align*}
Since $\Lambda^{-\frac 12}\theta_n \to \Lambda^{-\frac 12} \theta
$ strongly in $L^{\infty}$, by using Proposition \ref{prop_smoot1}, we obtain that
$\th$ solves \eqref{sqg}. 

Finally we remark that our solution $\th=\La f$ has an almost explicit form. By using 
\eqref{defn.f_new}, we have
\begin{align*}
f= \sum_{n=0}^{\infty} \sum_{j=1}^2
2 \sqrt{\frac{r_n}{5\lambda_{n+1} } } \left(P_{\le \mu_{n+1}}
\sqrt{C_0+R_j^o \frac{q_n}{r_n} }\right) \cos (5\lambda_{n+1} l_j\cdot x).
\end{align*}
The leading term is an almost explicit Fourier series (one can take $C_0$ large) and thus
our solution is nontrivial.

\end{proof}
\subsection*{Parameters}\label{para}
Throughout this paper, we fix parameters as follows. \smb{$\nu\ge 0$, $0<\gamma <\frac 32$,
$0<\beta<\min\{\frac 13, 3-2\gamma\}$},
\begin{equation}\label{par}
\la_n=\myceil{\la_0^{b^n}},\quad r_n=\la_n^{-\beta}, 
 \quad \mu_{n+1}= (\lambda_{n+1} \lambda_n)^{\frac 12}, \qquad n\in \N \cup \{0\},
\end{equation}
where $\myceil{\cdot }$ denotes the ceiling function. 
Here $\la_0\in \N$, $b=1^+$, will be chosen in Proposition \ref{const.qn}.
 The H\"{o}lder exponent in Theorem \ref{thm} is
\smb{$\al = \frac 12 + \frac {\beta}{2b} -\epsilon_0>\frac12$} by taking  first $b-1$ sufficiently small and then $\epsilon_0$ sufficiently small.  See also Appendix \ref{S:appC} for more explicit dependence of constants.

\section{Construction of $f_{n+1}$ }\label{sec3}
In this section we show that for given $q_n$, one can solve
the main piece in \eqref{eq.pert} up to a small error:
\begin{align}
\label{main.eq}
\La f_{n+1}\na^\perp f_{n+1}  +\na q_n \simp \text{small error.}
\end{align}

\subsection{Derivation of the leading order part}
Consider the ansatz ($f=f_{n+1}$)
\EQ{\label{ansatz}
f(x) = \sum_{l} a_l(x) \cos(\la l \cdot x), 
}
where the frequency $a_l$ is much smaller than $\la$ and the summation over $l$ is finite. 
%

\begin{lemm}[Leibniz]\label{Leib} Let $|l|=1$, $\la l \in \Z^2$, and $g(x) = a(x)\cos(\la l\cdot x)$. Then,
\[
\La g = \la g + (l\cdot \na a) \sin(\la l\cdot x) + (T^{(1)}_{\la l}a) \cos (\la l\cdot x) + (T^{(2)}_{\la l}a)\sin (\la l\cdot x),
\]
where 
\begin{align} \label{operator.T}
\widehat{T^{(1)}_{\la l}a}(k) =  \left(\frac{\left| \la l + k \right| + \left| \la l-k \right|}{2}-\la \right)\widehat{a}(k),\quad
\widehat{T^{(2)}_{\la l}a}(k) = i  \left(\frac{\left| \la l + k \right| - \left| \la l-k\right|}{2}-l\cdot k\right)\widehat{a}(k).
\end{align}
\end{lemm}
\begin{proof} The proof follows from a simple calculation using the following fact:
If $\widehat{T_m g}(k) =  m(k) \widehat g(k)$,  then for any $n\in \mathbb Z^2$,
 {$T_m( g(x ) e^{i n\cdot x})
=(T_{m_1} g)e^{i n \cdot x} $}, where $m_1(k)=m(k+n)$. 
\end{proof}

By using  Lemma \ref{Leib}, we have
\EQ{\label{nonlinear}
\La f\na^{\perp}f  
\simp \bmt{main} + \bmt{non-oscillatory\; error} + \bmt{oscillatory\; error},
}
where (below $l^{\perp}=(-l_2, l_1)^{\intercal}$ for $l=(l_1, l_2)^{\intercal}$)
\begin{align}
\bmt{main}
&= - \frac 14\la \sum_{l} (l\cdot \na)(a_l^2) l^{\perp}, \notag \\ \bmt{non-oscillatory\; error}
&= -\frac 12 \la \sum_{l} (T_{\la l}^{(2)}a_l)a_l l^{\perp} + \frac 12 \sum_{l} (T_{\la l}^{(1)} a_l) \na^{\perp}a_l,  \notag \\
\bmt{oscillatory\; error}
 =& \ \frac 12 \sum_{l} (l\cdot\na a_l + T_{\la l}^{(2)} a_l )  (\la a_l l^{\perp} \cos (2\la l\cdot x)+\na^{\perp} a_l  \sin (2\la l\cdot x))
 \label{Er1} \tag{\text{osc1}}\\
& - \frac 12  \sum_l (T_{ \la l}^{(1)}a_l) (\la a_l l^{\perp} \sin (2\la l\cdot x)-\na^{\perp} a_l \cos(2\la l\cdot x)) 
\label{Er2}
\tag{\text{osc2}}\\
& - \la\sum_{l\neq l'} (l\cdot\na a_l + T_{\la l}^{(2)}a_l)  a_{l'} (l')^{\perp} \sin (\la l\cdot x) \sin (\la l'\cdot x) 
\label{Er3}\tag{\text{osc3}}\\
& + \sum_{l\neq l'} (l\cdot \na a_l + T_{\la l}^{(2)} a_l ) \na^{\perp} a_{l'} \sin (\la l \cdot x) \cos (\la l'\cdot x) \label{Er4}\tag{\text{osc4}}\\
& -\la \sum_{l\neq l'} (T_{\la l}^{(1)}a_l) a_{l'} (l')^{\perp} \cos(\la l\cdot x) \sin (\la l'\cdot x) 
\label{Er5}\tag{\text{osc5}}\\
& + \sum_{l \neq l'} (T_{\la l}^{(1)}a_l) \na^{\perp} a_{l'} \cos(\la l\cdot x) \cos (\la l'\cdot x) 
\label{Er6}\tag{\text{osc6}}.
\end{align}
Note that the leading-order term $\la f\na^\perp f$ in $\La f\na^\perp f$ vanishes since $ \na^\perp \left( \frac{\la}2f^2\right)\simp 0$. 
\subsection{Matching}
We begin with a simple yet powerful lemma. 

\begin{lemm}[Algebraic Lemma]\label{decomp.id} For a given $Q\in C_0^\infty(\T)$, we have the decomposition identity
\EQN{
\sum_{j=1}^2 l_j^{\perp} (l_j\cdot \na)(\cR_j^o Q) \simp \na Q,
	}
where $l_1 = (\frac 35,\frac 45)^{\intercal}$, $l_2 = (1,0)^{\intercal}$, and the Riesz-type transforms $\cR_j^o$, $j=1,2$ are defined by 
\[
\widehat{\cR_1^o}(k_1,k_2) = \frac{25(k_2^2 -k_1^2)}{12 |k|^2}, \quad
\widehat{\cR_2^o}(k_1,k_2) = \frac{7(k_2^2 -k_1^2)}{12 |k|^2} +\frac {4k_1k_2}{|k|^2}.
\]
\end{lemm}

\begin{proof}
This follows from the identity $\sum_{j=1}^2 (l_j^{\perp}\cdot \na)(l_j\cdot \na)(\cR_j^o Q) = \De Q$.
\end{proof}




\begin{prop}\label{matching}
Set $l_j$ and $\cR_j^o$, $j=1,2$ as in Lemma \ref{decomp.id}.  For given
$q_n \in C_0^{\infty}(\T)$,  choose $C_0\ge 2$ to be a fixed constant and
\begin{align}
a_{j,n+1}^{\operatorname{perfect}}=2\sqrt{\frac{r_n}{5\la_{n+1}}}   
 \sqrt{ 
C_0+ \cR_{j}^o\frac{q_n}{r_n}},  \,
\label{defn.f}
 \end{align}
 where $(\lambda_{n+1},r_n)$ are taken as in \eqref{par}. Then
%
\begin{align}\label{mismatch.pr}
-\frac 14  
\cdot (5\la_{n+1}) \cdot \Bigl(\sum_{j=1}^{2}l_j^{\perp}(l_j\cdot\na) (a_{j,n+1}^{\operatorname{perfect}})^2\Bigr)+\na q_n 
\simp  0.
\end{align}
\end{prop}
\begin{proof} The proof follows from applying Lemma \ref{decomp.id} to $Q=q_n$. 
\end{proof}
We now choose 
\begin{align} \label{defn.f_new}
f_{n+1}(x) = \sum_{j=1}^2 a_{j,n+1}(x) \cos (5\la_{n+1} l_j \cdot x),
\qquad a_{j,n+1}= P_{\le \mu_{n+1}} a_{j,n+1}^{\operatorname{perfect}},
\end{align}
where $\widehat{P_{\le \mu_{n+1}} g}(k)
=\psi (\frac k {\mu_{n+1}} ) \widehat g(k)$, and
$\psi \in C_c^{\infty}(\mathbb R^2)$ satisfies $\psi (k)= 0$ for $|k|\ge 1$,
and $\psi (k)=1$ for $|k|\le \frac 12$.  We have 
$\Lambda f_{n+1} \nabla^{\perp} f_{n+1} +\nabla q_n \simp 
\text{small error}$. In the next section we estimate the errors.

\section{Error estimates for $q_{n+1}$ }\label{sec4}
\begin{prop}\label{const.qn} Given $\nu\geq 0$, $0<\ga<\frac32$, $0<\be <\min\left(\frac13,3-2\ga\right)$, there exists $b_0=b_0(\nu,\ga,  \be)$ such that for any $0<b-1<b_0$ we can find $\La_0=\La_0(\nu, \ga, \be, b)$ for which the following holds.  If $\la_0 \geq \La_0$ and $(f_\len, q_n)$ satisfies (below $\cR_j^o$ are the same as in Lemma \ref{decomp.id})
\begin{itemize}
    \item the frequencies of $f_\len$ and $q_n$ are localized to $\leq 6\la_n$ and $\leq 12\la_n$, respectively,
    \item $\norm{f_\len}_ {C^\al(\T)}\le 100$ and $\norm{q_n}_X\leq r_n$ where 
    \begin{align} \label{X_def}
    \norm{q}_X := \norm{q}_\infty + \sum_{j=1}^2\norm{\cR_j^o q}_\infty.
    \end{align}
\end{itemize}
 then there exists $q_{n+1}\in C_0^\infty(\T)$ solving \eqref{eq.pert} with
 frequency localized to $\le 12 \lambda_{n+1}$,  
 $f_{n+1}$ defined by \eqref{defn.f} satisfying
\EQ{\label{est.qn}
\norm{q_{n+1}}_X \leq r_{n+1}.
}
\end{prop}




\begin{proof}
Rewrite \eqref{eq.pert} as
\EQN{
\na q_{n+1} &\simp  
\underbrace{\La f_{n+1}\na^\perp f_{n+1}  + \na q_n}_{\text{Mismatch error}} 
+ \underbrace{\La f_{n+1}\na^\perp f_\len  + \La f_\len \na^\perp f_{n+1}  }_{\text{Transport error}} 
 \underbrace{-\nu\na\La^{\ga -1}f_{n+1}}_{\text{Dissipation error}} \\
& =: \na q_M + \na q_T + \na q_D.
}
Frequency localization of $q_{n+1}$ can be easily deduced from $q_M$, $q_T$, and $q_D$ which are defined below.
For convenience, we shall write $a_{j,n+1}$ as $a_j$ in the computation below.

\bigskip

\noindent\texttt{Mismatch error.} \ By \eqref{nonlinear}, we can further decompose the mismatch error 
as
\EQN{
\na q_M &\simp (\;\bmt{main}+ \na q_n)+ \bmt{non-oscillatory\; error} + \bmt{oscillatory\;error} \\
&\simp \na q_{M1} + \na q_{M2}+\na q_{M3}. 
}
By  Lemma \ref{Ap60.1}, $q_{M1}$ is defined as in \eqref{qM1} and satisfies
\begin{align} \label{qm1.est}
\norm{q_{M1}}_X
\lesssim r_n (\mu_{n+1}^{-1} \lambda_n)^2  \log \mu_{n+1}.
\end{align}

Note that both $\bmt{non-oscillatory\; error}$ and $\bmt{oscillatory\;error}$ have zero means, so we define 
\EQN{
q_{M2}
=\De^{-1}\na \cdot \bmt{non-oscillatory\;error}, \quad
q_{M3}
=\De^{-1}\na \cdot \bmt{oscillatory\;error}
}
in $C_0^\infty(\T)$. We postpone the estimate for $q_{M2}$ to Appendix, where Lemma \ref{comm.est} proves 
\EQ{\label{qm2.est}
\norm{q_{M2}}_X 
\lesssim  r_n \lambda_{n+1}^{-2} \mu_{n+1}^2 \log \mu_{n+1}.
}

Next we estimate $q_{M3}$. Denote $T_{n+1,j}^{(i)} = T_{{5\la_{n+1}l_j}}^{(i)}$ for $i,j=1,2$.  
By Lemma \ref{L_00a},  we have
\begin{align} 
&\norm{T_{n+1,j}^{(1)}a_j}_\infty
\lesssim\;  \lambda_{n+1}^{-1} \mu_{n+1}^2 \sqrt{\frac{r_n}{\lambda_{n+1}}},
\quad
\norm{T_{n+1,j}^{(2)}a_j}_\infty
\lesssim\; \lambda_{n+1}^{-2} \mu_{n+1}^3 \sqrt{\frac{r_n}{\lambda_{n+1}}}.
\label{T2.est}
\end{align}
Since all terms in \text{(oscillatory error)} have the frequency localized to $\sim \la_{n+1}$ provided that $48\la_n \leq \la_{n+1}$, the estimate for $q_{M3}$ easily follows from \eqref{T2.est}:
\EQN{
\norm{\De^{-1}\na \cdot \eqref{Er1}}_X
&\lec 
\sum_{j=1}^2 (\norm{\na a_j}_\infty + \norm{T_{n+1,j}^{(2)}a_j}_\infty)(\norm{a_j}_\infty + \la_{n+1}^{-1} \norm{\na^\perp a_j}_\infty)
\lesssim \bke{\frac{\la_n}{\la_{n+1}}} r_n,
}
\[
\norm{\De^{-1}\na \cdot \eqref{Er2}}_X
\lec \sum_{j=1}^2\norm{T_{{n+1,j}}^{(1)}a_j}_\infty (\norm{a_j}_\infty + \la_{n+1}^{-1} \norm{\na^\perp a_j}_\infty)
\lesssim
\left(\frac{\la_n}{\la_{n+1}}\right)r_n.
\]
The estimates for \eqref{Er3}-\eqref{Er6} are similar (using $2/\sqrt5\leq |l_1\pm l_2| \leq 4/\sqrt5$) and therefore
\EQ{\label{qm3.est}
\norm{q_{M3}}_X 
\lesssim \bke{\frac{\la_n}{\la_{n+1} }}  {r_n}.
}

Combining \eqref{qm1.est}, \eqref{qm2.est}, and \eqref{qm3.est} and using $b>1$, $\be<1$, we can find $\La_M= \La_M(\be, b)$ such that for any $\la_0\geq \La_M$, we get $q_M = q_{M1}+q_{M2}+q_{M3} \in C_0^{\infty}(\T)$ satisfying
\[
\norm{q_{M}}_X 
\le \frac 13 r_{n+1}.
\]

\noindent\texttt{Transport error.} Define
\[
\text{\smb{$
q_T = \De^{-1}\na \cdot (\La f_{n+1}\na^\perp f_\len  +\La f_\len \na^\perp f_{n+1}) \in C_0^\infty(\T)$}}. \label{linearerr}
\]
Since $\La f_{n+1}\na^\perp f_\len  +\La f_\len \na^\perp f_{n+1}$ is frequency-localized to $\sim \la_{n+1}$, using $\norm{f_\len}_{C^\al}\leq 100$, we get
\EQN{
\norm{q_T}_X 
\lesssim  \norm{f_{n+1}}_\infty (\norm{\na^\perp f_\len}_\infty
+\norm{\La f_\len}_\infty)
\leq C_\al \la_n^{1-\al}
\text{\smb{$\sqrt{\frac{r_n }{\la_{n+1}}}$}}
\le \frac 13 r_{n+1}
}
for some constant $C_\al>0$. We can find $\La_T=\La_T(\be, b)$ such that for any $\la_0\geq \La_T$ the last inequality holds since $b>1$ and $\be<\frac 15$. 
\bigskip

\noindent\texttt{Dissipation error.} \ We define $q_D = -\nu \La^{\ga-1} f_{n+1}\in C_0^\infty(\T)$ which satisfies
\EQN{
\norm{q_D}_X 
\le C_2 \la_{n+1}^{\ga-1} \norm{f_{n+1}}_\infty
\le 5C_2 \la_{n+1}^{\ga-1} 
\text{\smb{$\sqrt{\frac{r_n}{\la_{n+1}}}$}}
\le \frac 13 r_{n+1}, \label{diserr}}
for some  $C_2=C_2(\nu, \ga)>0$. Since
$\be < 3-2\ga$, we can find sufficiently small $b_0=b_0(\nu, \ga,\be)$ such that for any $1<b<b_0+1$ there exists $\La_D=\La_D(\nu, \ga,\be,b)$ which leads the last inequality for any $\la_0\geq \La_D$.

\medskip

Collecting the estimates, we obtain   $\norm{q_{n+1}}_X \leq r_{n+1}$
if $\lambda_0>\La_0 = \max(\La_M, \La_T, \La_D)$.
\end{proof}


\appendix
\section{Non-oscillatory error estimate }
\label{appendix}

\begin{lemm} \label{L_m90}
Suppose $a:\, \mathbb T^2\to \mathbb R$ with $\operatorname{supp}
(\widehat{a})\subset \{ |k| \le \mu \}$ and $\mu \ge 10$.  Then for any
Riesz-type operator $\mathcal R$ we have
\smb{$
\| \mathcal R a \|_{\infty} \lesssim  \| a \|_{\infty} \log \mu$}.

\end{lemm}
\begin{proof}
WLOG we can assume $\bar a=0$. By using LP-decomposition, 
splitting into low and high frequencies and choosing \smb{$J=2\log \mu$}, we obtain 
\smb{$\|\mathcal R a \|_{\infty} \lesssim  (J+3) \| a\|_{\infty} + 2^{-J} \| \nabla a \|_{\infty}
\lesssim (J+3+ 2^{-J} \mu) \|a\|_{\infty} \lesssim \|a\|_{\infty} \log \mu$}.  
\end{proof}
\noindent
We now state two useful facts. 
Assume \smb{$f \in C^{\infty}(\mathbb T^2)$} and \smb{$K \in L^1(\mathbb R^2)$}
with \smb{$m(\xi) = \int_{\mathbb R^2} K(z) e^{-i  \xi \cdot z} dz$}.  
Then\footnote{Here and below we still 
denote by $f$ its periodic extension to all of $\mathbb R^2$.} 
\begin{align}
\text{\smb{$(T_m f)(x) := \sum_k m(k) \hat f (k) e^{ik \cdot x}
= \int_{\mathbb R^2} K(z) f(x-z) dz$},}
\quad \text{\smb{$\|T_m f \|_{L^p_x(\mathbb T^2)}
\le \| K\|_{L_x^1(\mathbb R^2)} \| f \|_{L^p_x(\mathbb T^2)} $,}}
\; \text{\smb{$\forall\, 1\le p\le \infty$}}. \label{fact.0.01}
\end{align}
Assume \smb{$f, g\in C^{\infty}(\mathbb T^2)$} and 
\smb{$K \in L^1(\mathbb R^2 \times
\mathbb R^2)$} with \smb{$m(\xi,\eta)=\int_{\mathbb R^2\times
\mathbb R^2}
K(z_1,z_2)e^{-i \xi \cdot z_1 -i \eta \cdot z_2} dz_1 dz_2$}. Then
\begin{align}
\text{\smb{$
T_{m}(f,g)(x) := \sum_k \Bigl( \sum_{k^{\prime} \in \mathbb Z^2}
m(k^{\prime}, k-k^{\prime}) \hat f(k^{\prime}) \hat g(k-k^{\prime}) \Big)e^{i k\cdot x} 
=  \int_{\mathbb R^2\times \mathbb R^2}
K(z_1, z_2) f(x-z_1) g(x-z_2) dz_1 dz_2,
$}}  \label{fact.0.02}
\end{align}
and consequently
\smb{$
\| T_m(f,g) \|_{L_x^r(\mathbb T^2)} \le 
 \| K\|_{L_x^1(\mathbb R^2\times \mathbb R^2)}
\| f \|_{L_x^p(\mathbb T^2)} \| g \|_{L_x^q(\mathbb T^2)}$}
for any \smb{$1\le r,p,q\le \infty$} with \smb{$\frac 1r = \frac 1p +\frac 1 q$}.

\begin{lemm} \label{L_00a}
Assume \smb{$b_0:\mathbb T^2\to \mathbb R$} with
$\operatorname{supp}(\widehat{b_0}) \subset\{ |k | \le \mu \}$ and 
$10\le \mu \le  \frac 12 \lambda$. Then (see \eqref{operator.T})
\begin{align*}
\text{\smb{$
\| T^{(1)}_{\lambda l} b_0\|_{\infty}  \lesssim \lambda^{-1} \mu^2 \|b\|_{\infty}$}},
\quad 
\text{\smb{$
\| T^{(2)}_{\lambda l} b_0\|_{\infty}
\lesssim \lambda^{-2} {\mu^3} \| b_0\|_{\infty}$}},
\quad
\text{\smb{$
\| \Delta^{-1} \nabla T^{(2)}_{\lambda l} b_0\|_{X}
\lesssim  \|b_0\|_{\infty} \lambda^{-2} {\mu^2} \log \mu $}}.
\end{align*}
\end{lemm}
\begin{proof}
We show only the first one as the rest are
similar. 
Choose \smb{$\phi_1 \in C_c^{\infty}(\mathbb R^2)$} such that
\smb{$\phi_1(\xi) \equiv 1$} for 
\smb{$|\xi| \le 1$} and \smb{$\phi_1(\xi) \equiv 0$}
for \smb{$|\xi|\ge 1.1$}. Denote {$\phi_2(z)= |l+z|+|l-z|-2$} and note that 
for \smb{$|z| \le \frac 23$} we have 
\smb{$\phi_2(z) = \sum_{i,j=1}^2 h_{ij}(z) z_i z_j$} for some
$h_{ij} \in C^{\infty}$. By \eqref{fact.0.01} it suffices to show
 \smb{$\|F\|_{L_x^1(\mathbb R^2)}
\lesssim \lambda^{-2} \mu^2$} for
 \smb{$F(x)=
\int_{\mathbb R^2} \phi_2(\lambda^{-1} \xi)
\phi_1(\mu^{-1} \xi) e^{i \xi \cdot x} d\xi$}.  This follows from a change of
variable \smb{$\mu^{-1}\xi \to \xi$} and integration by parts. 
For the third estimate one can extract an extra gradient from the symbol and then
use Lemma \ref{L_m90}.
\end{proof}

\begin{lemm} \label{L_00b}
Let \smb{$\operatorname{supp}(\widehat{b_0}) \subset\{ |k | \le \mu \}$},
\smb{$\mu \le  \frac 12 \lambda$}. Then for\footnote{Here \smb{$\mathcal F^{-1}$} denotes
Fourier inverse transform on \smb{$\mathbb R^2\times \mathbb R^2$}. See 
\eqref{fact.0.02}.}
 some \smb{$K_i=\mathcal F^{-1}(m_i)$} with
\smb{$\|K_i\|_{L^1(\mathbb R^4)} \lesssim 1$}, we have
\begin{align*}
\text{\smb{$
b_0 T_{\lambda l}^{(2)} b_0  =\frac {\mu^2}{\lambda^2} \sum_{i=1}^2 \partial_{x_i}T_{m_i} 
(b_0,b_0)$
}},
\quad
\text{\smb{$
 (T_{\lambda l}^{(1)} b_0 )\partial_{x_1} 
b_0  =\frac {\mu^2}{\lambda} \sum_{i=3}^4 \partial_{x_i}T_{m_i} (b_0,b_0)$}},
\quad 
\text{\smb{$
(T_{\lambda l}^{(1)} b_0 )\partial_{x_2} 
b_0  =\frac {\mu^2}{\lambda} \sum_{i=5}^6 \partial_{x_i}T_{m_i} (b_0,b_0)$}}.
\end{align*}
\end{lemm}

\begin{proof}
Observe that for \smb{$|z|\le \frac 23$}, \smb{$\phi(z) =|l+z| -|l-z| -2 l \cdot z =\sum_{i,j,k=1}^2
h_{ijk}(z) z_i z_j z_k$} for some \smb{$h_{ijk} \in C^{\infty}$}. 
Choose \smb{$\phi_1 \in C_c^{\infty}(\mathbb R^2)$} such that 
\smb{$\phi_1(\xi) \equiv 1$} for 
\smb{$|\xi| \le 1$} and \smb{$\phi_1(\xi) \equiv 0$}
for \smb{$|\xi|\ge 1.1$}.  By using parity of $\phi$, we have
\begin{align*}
&\text{\smb{$
\widehat{b_0 T_{\lambda l}^{(2)} b_0}
(k)  = \frac i 4 \lambda  \sum_{k^{\prime} \in \mathbb Z^2}
( \phi(\lambda^{-1} k^{\prime}) - \phi( \lambda^{-1} (k^{\prime}-k) ))
\widehat {b_0} (k^{\prime} ) \widehat{b_0}(k-k^{\prime} ) $}}\notag \\
&\text{\smb{$
 =-\frac i 4 \sum_{k^{\prime} \in \mathbb Z^2}
\int_0^1 k \cdot (\nabla \phi)( \lambda^{-1}(k^{\prime} -\theta k) )
d\theta \phi_1(\mu^{-1} k^{\prime})  \phi_1(\mu^{-1}(k-k^{\prime}) )
\widehat{b_0}(k^{\prime}) \widehat{b_0}(k-k^{\prime}).$}}
\end{align*}
Note that \smb{$ (\nabla \phi) (\frac {k^{\prime}-\theta k} {\lambda} )
\phi_1(\frac{ k^{\prime}} {\mu} ) \phi_1(\frac {k-k^{\prime}} {\mu} )
=\lambda^{-2} \sum_{1\le i,j\le 2} \tilde h_{ij}(\frac{k^{\prime}-\theta k}{\lambda} )
(k^{\prime}-\theta k)_i (k^{\prime}-\theta k)_j
\phi_1(\frac{ k^{\prime}} {\mu} ) \phi_1(\frac {k-k^{\prime}} {\mu} )$
} where \smb{$\tilde h_{ij}
\in C^{\infty}_c (\mathbb R^2)$}. 
The result then follows from \eqref{fact.0.02} by checking the $L^1$ bound of the kernel.
The case for \smb{$T_{\lambda l}^{(1)}$} is similar.
\end{proof}

\begin{lemm}\label{comm.est}
Define $T_{n+1,j}^{(i)}=T_{5\la_{n+1}l_j}^{(i)}$, $a_j = a_{j,n+1}$, and $l_j$, $i,j=1,2$, as in \eqref{operator.T} and \eqref{defn.f}. Suppose that the assumptions on $(f_\len, q_n)$ in Proposition \ref{const.qn} hold. 
Then, we have
\EQN{
\norm{\De^{-1}\na \cdot ( (T_{n+1,j}^{(1)}a_j) \na^\perp a_j )}_X
+\norm{\De^{-1}\na \cdot (5\la_{n+1} (T_{n+1,j}^{(2)}a_j) a_j l^\perp_j)}_X
\lesssim  r_n \lambda_{n+1}^{-2} \mu_{n+1}^2 \log \mu_{n+1}.
}
\end{lemm}

\begin{proof}We only treat the second term as the others are similar.
By Lemma \ref{L_00b} and \ref{L_m90}, we have
\begin{align*}
\norm{\De^{-1}\na \cdot (5\la_{n+1} (T_{n+1,j}^{(2)}a_j) a_j l^\perp_j)}_X
\lesssim (\log \mu_{n+1}) \lambda_{n+1}
(\frac {\mu_{n+1} }{\lambda_{n+1} })^2
\frac {r_n}{\lambda_{n+1}} 
\lesssim r_n (\frac {\mu_{n+1} }{\lambda_{n+1} })^2 \log \mu_{n+1}.
\end{align*}

\end{proof}

\begin{lemm}[Estimate of $q_{M1}$]   \label{Ap60.1}
We have $\|q_{M1} \|_X \lesssim (\log \mu_{n+1})
(\mu_{n+1}^{-1} \lambda_n)^2 r_n$. 
\end{lemm}
\begin{proof}
To ease the notation we write $a_j^p=2 \sqrt{\frac{r_n}{\lambda_{n+1}}}
 \sqrt{ C_0 +  \mathcal R_j^o \frac {q_n} {r_n} }$ and $a_j = P_{\le \mu_{n+1}}
 a_j^p$. By using a fattened frequency projection $\tilde P_{\le \mu_{n+1} }$ which
is frequency localized to $\{ |k| \le 4\mu_{n+1} \}$, 
 we have 
 \begin{align*}
& -\frac 14 \cdot (5\lambda_{n+1}) \cdot
  \sum_{j=1}^2 l_j^{\perp} (l_j\cdot  \nabla ) a_j^2 + \nabla q_n - \nabla q_{M1} \notag \\
 = & - \frac 54  \lambda_{n+1} \sum_{j=1}^2 l_j^{\perp} (l_j \cdot \nabla)
  \tilde P_{\le \mu_{n+1}} (  ( P_{\le \mu_{n+1} } a_j^p )^2 ) 
  + \nabla q_n - \nabla  q_{M1} \notag \\
  =& - \frac 54  \lambda_{n+1} \sum_{j=1}^2 l_j^{\perp} (l_j \cdot \nabla)
  \tilde P_{\le \mu_{n+1}} \biggl(   -2 a_j^p P_{>\mu_{n+1} } a_j^p 
  + (P_{>\mu_{n+1} } a_j^p)^2 \biggr)  - \nabla  q_{M1}\simp 0. 
  \end{align*}
  Thus we can solve $q_{M1} \in C_0^{\infty}(\mathbb T^2)$ as 
  \begin{align}\label{qM1}
q_{M1} = -\frac 54 \lambda_{n+1} \sum_{j=1}^2
\Delta^{-1} \nabla \cdot \Bigl( l_j^{\perp} (l_j \cdot \nabla)
  \tilde P_{\le \mu_{n+1}} \biggl(   -2 a_j^p P_{>\mu_{n+1} } a_j^p 
  + (P_{>\mu_{n+1} } a_j^p)^2 \biggr)  \Bigr).
  \end{align}
  Note that $q_{M1}$ is frequency localized to $\{|k| \le 4\mu_{n+1}\}$. 
  By Lemma \ref{L_m90}, we obtain
\begin{align*}
\| q_{M1} \|_X
&\lesssim \log \mu_{n+1} \cdot \lambda_{n+1}
\sum_{j=1}^2 \| a_j^p \|_{\infty} \| P_{>{\mu_{n+1} }} a_j^p\|_{\infty} \lesssim 
 \log \mu_{n+1} \cdot (\mu_{n+1}^{-1} \lambda_n)^2 r_n.
\end{align*}

\end{proof}

\section{Some technical estimates}

\begin{prop} \label{prop_smoot1}
Let $\mathcal R=\mathcal R_j$, $j=1,2$. Assume $\phi \in H^3$ and $\theta \in \dot H^{-\frac 12}$ ($\overline \theta=0$). Then we have
\begin{align*}
\| [\mathcal R, \phi ] \theta \|_{\dot H^{\frac 12} } \lesssim 
\|  \phi\|_{\dot H^3}
\| \theta \|_{\dot H^{-\frac 12} }.
\end{align*}
\end{prop}
\begin{proof}
Denote $m(k)=\frac {k_1}{|k|}$. 
 It suffices to show that
\begin{align} \label{te_456}
\| \sum_{k^{\prime} \ne 0,k} |k|^{\frac 12}
 (m(k) -m(k^{\prime})) \widehat{\phi}(k-k^{\prime})
\widehat{\theta}(k^{\prime}) \|_{l_k^2} \lesssim 
\| |k|^3 \widehat {\phi}(k) \|_{l_k^2} \| |k|^{-\frac 12} \widehat {\theta}(k) \|_{l_k^2}.
\end{align}
If $|k^{\prime}| \lesssim |k-k^{\prime}|$, then $|k| \lesssim |k-k^{\prime}|$, and
\begin{align*}
\text{LHS of \eqref{te_456}}
\lesssim \| \sum_{k^{\prime} \ne 0, k}
|k-k^{\prime}| |\widehat{\phi} (k-k^{\prime})| \cdot |k^{\prime}|^{-\frac 12}
|\widehat{\th}(k^{\prime})|
\|_{l_k^2} \lesssim  \text{RHS of \eqref{te_456}}.
\end{align*}
If $|k-k^{\prime}|\ll |k|$, then $|k| \sim |k^{\prime}|$, and 
it suffices to use $|m(k)-m(k^{\prime})| \lesssim |k-k^{\prime}| 
(|k^{\prime}|+|k|)^{-1}$.

\end{proof}

%
\begin{proof}[Proof of Theorem \ref{thm2}]
The point is to use the weak formulation (below $\langle, \rangle $ denotes $L^2$-inner
product in $(t,x)$, and $\psi$ is a time-dependent test function)
$$\langle \partial_t \theta_n, \psi \rangle + \frac 12 \langle 
\La^{-\frac 12} \theta_n, \La^{\frac 12} [\mathcal R^{\perp}, \nabla \psi ] \theta_n \rangle
+ \nu \langle \La^{-\frac 12} \theta_n ,\La^{\gamma+\frac 12} \psi \rangle =0.
$$
By using the above together with Proposition \ref{prop_smoot1}, we have\footnote{Here 
$t$ belongs to an arbitrary compact interval.}
 $\| \partial_t \theta_n \|_{L_t^1 \dot H^{-8}} \lesssim 1$. Fix any $0\ne k \in \mathbb Z^2$. 
We have \smb{$\| \partial_t \widehat{\theta_n}(k,t) \|_{L_t^1} \lesssim |k|^8$} and 
\smb{$\|\widehat{\theta_n}(k, t)\|_{L_t^2} \lesssim |k|^{-s} $} which implies that for a subsequence
(and using a diagonal argument) \smb{$ \| \widehat{\theta_{n_l}}(k,t) -\widehat f(k,t) \|_{L_t^2}
\to 0$} for any fixed $k$. Using $\theta_n \in L_t^2 \dot H^s$,  one obtains $\theta_{n_l} \to f$ in $L_t^2
\dot H^{-\frac 12}$.  Since \smb{$\|\La^{\frac 12} [\mathcal R^{\perp}, \nabla \psi ] (\theta_n
-f) \|_2 \lesssim \| \theta_n -f \|_{\dot H^{-\frac 12}}$}, $f$ is clearly the desired weak
solution. 
\end{proof}

\section{Bookkeeping of various parameters}\label{S:appC}
In this appendix we sketch how the choice of various parameters in \eqref{par} take effect
on various error terms and the regularity of the weak solution. 
Recall that (observe from below $\log \mu_{n+1} \sim \log \lambda_n$)
\begin{align}\notag
\la_n=\myceil{\la_0^{b^n}},\quad r_n=\la_n^{-\beta}, 
 \quad \mu_{n+1}=(\lambda_n \lambda_{n+1})^{\frac 12},  \quad 
 \text{\smb{$\al = \frac 12 + \frac {\beta}{2b} -\epsilon_0>\frac12$}}. 
\end{align}

\texttt{Mismatch error} \qquad  \smb{$ r_n 
\frac {\lambda_n} {\lambda_{n+1} }  \log\la_n \ll r_{n+1}\iff
\la_n^{(b-1)(\be-1)}\log \la_n\ll 1. $}

\texttt{Transport error} \quad\  \smb{$\la_n^{1-\al}\sqrt{\frac{r_n}{\la_{n+1}}}\ll r_{n+1}\iff \la_n^{1-\alpha-\frac 12 \beta-\frac 12 b+b\beta}\ll1.$}

\texttt{Dissipation error} \  \smb{$\la_{n+1}^{\ga-1}\sqrt{\frac{r_n}{\la_{n+1}}}\ll r_{n+1}\iff \la_{n+1}^{\gamma-\frac 32+\beta
-\frac{\beta}{2b} }\ll 1.$}

\texttt{$C^\al$-regularity}\qquad\  \smb{$\la_{n+1}^\al\sqrt{\frac{r_n}{\la_{n+1}}}\ll 1 \iff \la_{n+1}^{\alpha-\frac 12-\frac 1{2b} \beta}\ll 1.$}

Now one can take $\alpha=\frac 12 +\frac {\beta}{2b} $ to do a limiting computation.
From the transport error we obtain (the limiting condition)
\begin{align*}
1-\alpha-\frac 12 \beta-\frac 12 b +b \beta=\frac {1-b}{2b}
(b-\beta(2b+1)) \Rightarrow \beta<\frac 13.
\end{align*}
From the dissipation error we obtain $\frac {\beta}2 <\frac 32 -\gamma$. 

\section*{Acknowledgments}
X. Cheng was supported by the International Doctoral Fellowship (IDF) from the University of British Columbia, Canada. H. Kwon was partially supported by NSERC grant 261356-13 (Canada) and the NSF grant No.DMS-1638352. D. Li was supported in part by Hong Kong RGC grant GRF 16307317
and 16309518. 

\bibliographystyle{abbrv}

\end{document}